\documentclass[11pt]{article}
\usepackage[affil-it]{authblk}
\usepackage[letterpaper,margin=0.9in]{geometry}
\usepackage{amsmath,amssymb,amsthm,mathtools}
\usepackage{microtype}
\usepackage{booktabs}
\usepackage{enumitem}
\usepackage[colorlinks=true,linkcolor=blue!50!black,citecolor=blue!50!black,urlcolor=blue!50!black]{hyperref}
\usepackage{xcolor}

\newtheorem{theorem}{Theorem}[section]
\newtheorem{proposition}[theorem]{Proposition}

\theoremstyle{remark}

\newcommand{\R}{\mathbb{R}}
\newcommand{\Id}{\operatorname{Id}}
\newcommand{\dd}{\,\mathrm{d}}
\newcommand{\cG}{\mathcal{G}}
\newcommand{\cC}{\mathcal{C}}
\newcommand{\Om}{\Omega}
\newcommand{\Th}{\Theta}
\newcommand{\la}{\lambda}
\newcommand{\eps}{\varepsilon}
\newcommand{\ip}[2]{\left\langle #1,#2\right\rangle}
\newcommand{\norm}[1]{\left\lVert #1\right\rVert}

\newcommand{\abs}[1]{\left\lvert #1\right\rvert}

\title{An inverse scattering problem for the nonlinear Schr\"odinger equation under partial harmonic confinement
}

\author[1]{Pranav Kumar\thanks{kumarp@mail.tsinghua.edu.cn}}
\affil[1]{Yau Mathematical Sciences Center, Tsinghua University, Beijing, China}

\author[1]{Li Li \thanks{lili19940301@mail.tsinghua.edu.cn}}

\date{}

\begin{document}
	\maketitle
	\begin{abstract}
		We study an inverse problem for a nonlinear Schr\"odinger equation with a harmonic potential in some spatial directions and free propagation in the remaining directions. The nonlinearity has an unknown power and a real, spatially dependent coefficient. For small incoming data in a weighted energy space, we construct the wave operator and show that it determines both the power and the coefficient when the latter depends either on the confined variables or on the free variables. The proof uses the first nonlinear term of the wave operator and concentrated product data. We explain why dependence on the free variables requires a stronger lower bound on the power. We also discuss the outgoing wave operator, the small-data scattering operator, and coefficients depending on all spatial variables.
	\end{abstract}
	
	\section{Introduction}
	Nonlinear Schr\"odinger equations (NLS) with harmonic potentials arise naturally
	in the study of Bose--Einstein 
	condensates, where the potential
	models trapping; see \cite{JosserandPomeau,PitaevskiiStringari}. In a partially confined
	configuration, the condensate remains 
	trapped in some spatial directions
	while it is allowed to disperse in the others. Scattering theory and existence of wave operators for NLS under partial harmonic confinement, with constant nonlinear coefficient, were established in
	\cite{AntonelliCarlesDrumond}. Spatially dependent nonlinearities also 
	appear in the Gross--Pitaevskii
	description, where the interaction strength is related to the scattering
	length and may vary in space; see
	\cite{SakaguchiMalomed,Yamazaki}.
	
	In this paper, we consider an inverse 
	scattering problem for the nonlinear
	Schr\"odinger equation under partial harmonic confinement. We are interested in recovering both the nonlinear exponent and the spatially
	dependent coefficient from the small-data incoming wave operator.
	
	More precisely, we consider
	\begin{equation}\label{eq:NLS}
		i\partial_t u
		=
		Hu+\lambda|u|^{2\sigma}u,
		\qquad
		(t,x,y)\in\R\times\R^n\times\R^m,
	\end{equation}
	where $m=d-n$, $1\le n\le d-1$, and
	\[H=H_x-\tfrac12\Delta_y,\quad H_x=-\tfrac12\Delta_x+\tfrac12|x|^2.
	\]
	Here $x\in\R^n$ denotes the confined variables, while
	$y\in\R^m$ denotes the unconfined variables. We consider separately the
	two cases
	\[
	\lambda=\lambda(x)
	\qquad\text{and}\qquad
	\lambda=\lambda(y).
	\]
	Our goal is to recover $\sigma$ and $\lambda$ from the small-data incoming
	wave operator associated with \eqref{eq:NLS}.
	
	To formulate the inverse problem, we work in the weighted energy space
	introduced in \cite{AntonelliCarlesDrumond} (see \eqref{eq:Sigma-def} below).
	For sufficiently small $f\in\Sigma$, let $u$ be the solution of
	\eqref{eq:NLS} satisfying
	\[
	\norm{e^{itH}u(t)-f}_{\Sigma}\longrightarrow0,
	\qquad t\to-\infty.
	\]
	The corresponding incoming wave operator is defined by
	\begin{equation}\label{eq:wave}
		\Omega^-_{\lambda,\sigma}(f):=u(0).
	\end{equation}
	We ask whether the knowledge of $\Omega^-_{\lambda,\sigma}$ for small
	incoming data determines the nonlinear exponent $\sigma$ and the
	coefficient $\lambda$. The following main result gives an affirmative answer to this question for both
	$\lambda=\lambda(x)$ and $\lambda=\lambda(y)$, under suitable assumptions on the nonlinear exponent and the coefficient.
	
	\begin{theorem}\label{thm:main}
		Let $\lambda_j$ be real-valued, nonzero coefficients, $j=1,2$, and let $\sigma_j>0$.
		\begin{enumerate}
			\item Suppose $\lambda_j=\lambda_j(x)\in W^{1,\infty}(\R^n)$ and
			\[\sigma_j>\frac{2d}{(d+2)m},\qquad \sigma_j<\frac2{d-2}\quad(d\ge3).\]
			\item Suppose $\lambda_j=\lambda_j(y)\in W^{1,\infty}(\R^m)$ and
			\[\sigma_j>\frac{4d}{(d+2)m},\qquad \sigma_j<\frac2{d-2}\quad(d\ge3).\]
		\end{enumerate}
		In each case assume the indicated interval is nonempty. If the two incoming wave operators agree on a common neighborhood of zero, then $\sigma_1=\sigma_2$ and $\lambda_1=\lambda_2$ (as their Lipschitz representatives).
	\end{theorem}
	The nonzero assumption is necessary: if $\lambda\equiv0$, then $\Omega^-_{0,\sigma}$ is the identity map for every $\sigma$.

	\subsection{Connection with earlier literature}
	For stationary linear Schr\"odinger scattering, one illuminates a potential at a fixed frequency (energy) and observes its scattering amplitude or scattering matrix. Recovering the potential from these data is closely related to the Calder\'on inverse boundary value problem (see \cite{SylvesterUhlmann, UhlmannVasy, PaivarintaSaloUhlmann}).
	
	Here the equation is evolutionary and nonlinear. The data are a map between asymptotic states, not measurements at one stationary frequency.
	
	The scattering theory for nonlinear Schr\"odinger equations with partial
	dispersion has been studied in several settings. Tzvetkov and Visciglia
	\cite{TzvetkovVisciglia} considered small-data scattering on product spaces
	$\R^m\times M$, while Hani and Pausader \cite{HaniPausader} studied the quintic
	defocusing NLS on $\R\times\mathbb T^2$. For NLS under partial harmonic
	confinement, Antonelli, Carles, and Drumond Silva \cite{AntonelliCarlesDrumond} established
	global-in-time Strichartz estimates and constructed wave operators. For related scattering results in similar settings, see, for instance, \cite{CarlesGallo,HaniThomann,ArdilaCarles} and the references therein.
	
	Inverse problems for nonlinear dispersive equations have been extensively
	studied. We refer to \cite{Weder97,Weder01b,CarlesGallagher,Watanabe,PausaderStrauss} and the references therein. The study of inverse scattering problems for 
	nonlinear Schr\"odinger equations goes back to 
	Strauss \cite{Strauss}, who showed that the coefficient and the power of the nonlinearity can be recovered from the small-data scattering map. This approach was further developed by Weder \cite{Weder97,Weder01a,Weder01b}, who considered the recovery of the potential and nonlinear terms from the scattering operator. Related
	inverse scattering problems for nonlocal and Hartree-type nonlinearities
	were studied in \cite{SasakiWatanabe,Sasaki07,Sasaki08,Sasaki12,Sasaki24}.
	Watanabe \cite{Watanabe} studied the recovery of a spatially dependent
	coefficient from the scattering map for nonlinear Schr\"odinger equations
	with decaying nonlinearities. Murphy \cite{Murphy}
	considered nonlinear Schr\"odinger equations 
	with nonlinearities of the form $\alpha(x)|u|^p u$ and showed that the nonlinear exponent and 
	the spatially dependent coefficient can be recovered from the scattering
	map. More recently, Killip, Murphy, and Vi{\c{s}}an \cite{KMV} showed that,
	for a general class of nonlinearities, both the wave operator and the scattering map determine the nonlinearity. The same authors \cite{KMV25}
	extended this result to a broad class of gauge-invariant, intercritical nonlinear Schr\"odinger equations.

	Another approach to inverse problems for nonlinear dispersive equations is based on the construction of approximate solutions. For nonlinear
	wave equations, S\'a Barreto and Stefanov
	\cite{SaBarretoStefanov} recovered a cubic nonlinearity in the weakly nonlinear regime. In the nonlinear Schr\"odinger setting, Hogan, Murphy and Grow \cite{HoganMurphyGrow} considered the recovery of a cubic
	spatially dependent nonlinearity, while Lee and Yu \cite{LeeYu} extended this approach to generalized higher-order Schr\"odinger equations. More recently, Lee and Pavlovi\'c \cite{LeePavlovic} studied the recovery of a spatially dependent coefficient for the cubic
	nonlinear Schr\"odinger equation in Sobolev spaces.
	
	The point of the present work is to recover both an unknown power and a variable nonlinear coefficient from the incoming wave operator for the partially confined NLS. The main feature of this model is that, the harmonic directions do
	not decay in time, while the free directions do. This forces us to distinguish the $x$ dependence from the $y$ dependence, for which a differentiated coefficient creates a factor growing like $t$.
	
	\subsection{Organization}
	
	The rest of this paper is organized in the following way. Section 2 records the linear propagator, commuting fields, mixed Strichartz norms, and the estimates used throughout. In Section 3, we prove a small-data forward result for $\lambda(x)$ and then recover $\sigma$ and $\lambda(x)$ from its wave operator. Section 4 explains the extra time-weighted estimate needed for $\lambda(y)$ and reconstructs this coefficient. Section 5 compares the two cases and gives formulas for outgoing and scattering operators, followed by the extension to $\lambda(x,y)$.
	\medskip
	
	\noindent \textbf{Acknowledgments.} L.L. would like to thank Professor Gunther Uhlmann for helpful discussions.
	
	\section{Preliminaries}
	
	\subsection{The linear evolution}
	
	The self-adjoint operator $H$ generates the unitary group
	\begin{equation}
		U(t)=e^{-itH}=e^{-itH_x}e^{it\Delta_y/2}.
		\label{eq:linear-group}
	\end{equation}
	For $t\ne0$, the free Schr\"odinger has the familiar integral formula
	\begin{equation}\label{eq:free}
		(e^{it\Delta_y/2}g)(y)=(2\pi it)^{-m/2}\int_{\R^m}e^{i|y-z|^2/(2t)}g(z)\,dz.
	\end{equation}
	For $t\notin\pi\mathbb Z$, the harmonic oscillator has the Mehler formula
	\begin{equation}\label{eq:mehler}
		(e^{-itH_x}\psi)(x)=(2\pi i\sin t)^{-n/2}\int_{\R^n}
		\exp\!\left[\frac{i}{2\sin t}\bigl((|x|^2+|z|^2)\cos t-2x\cdot z\bigr)\right]\psi(z)\,dz.
	\end{equation}
	It is well-known that, the free Schr\"odinger disperses as $|t|\to\infty$, whereas perturbations of the harmonic oscillator lead to KAM type results without corresponding large-time decay \cite{AntonelliCarlesDrumond}.
	
	The weighted energy space is
	\begin{equation}
		\Sigma
		:=
		\left\{
		f\in L^2(\R^d):
		|x|f,\;|y|f,\;\nabla_{x,y}f\in L^2(\R^d)
		\right\},
		\label{eq:Sigma-def}
	\end{equation}
	equipped with the norm
	\begin{equation}
		\norm{f}_{\Sigma}
		=\norm{\nabla_{x,y}f}_2+\norm{|x|f}_2+\norm{|y|f}_2.
		\label{eq:Sigma-norm}
	\end{equation}
	This is indeed a norm equivalent to the version with an extra $\norm f_2$: the uncertainty estimate controls $\norm f_2$ by $\norm{\nabla f}_2+\norm{|(x,y)|f}_2$. The restriction $m>0$ allows scattering through the unconfined directions.
	
	\subsection{Commuting vector fields}
	
	Set $A_0=\Id$ and
	\begin{align}
		A_1(t)&=x\sin t-i\cos t\,\nabla_x,
		&A_2(t)&=x\cos t+i\sin t\,\nabla_x,\notag\\
		A_3(t)&=-i\nabla_y,
		&A_4(t)&=y+it\nabla_y.
		\label{eq:vector-fields}
	\end{align}
	
	By \cite[Lemma~4.1]{AntonelliCarlesDrumond}, these are the Heisenberg evolutions of $-i\nabla_x,x,-i\nabla_y,y$, respectively. More precisely,
	\begin{align}
		A_1(t)&=e^{-itH}(-i\nabla_x)e^{itH},
		&A_2(t)&=e^{-itH}xe^{itH},\notag\\
		A_3(t)&=e^{-itH}(-i\nabla_y)e^{itH},
		&A_4(t)&=e^{-itH}ye^{itH}.
		\label{eq:Heisenberg}
	\end{align}
	Consequently,
	\begin{equation}
		[i\partial_t-H,A_j(t)]=0,
		\qquad
		A_j(t)U(t-s)=U(t-s)A_j(s).
		\label{eq:commutation}
	\end{equation}
	The pair $(A_1,A_2)$ is an orthogonal rotation of $(x,-i\nabla_x)$:
	\[
	\begin{pmatrix}A_2\\ A_1\end{pmatrix}
	=\begin{pmatrix}\cos t&-\sin t\\ \sin t&\cos t\end{pmatrix}
	\begin{pmatrix}x\\ -i\nabla_x\end{pmatrix}.
	\]
	Moreover, $y=A_4+tA_3$.  Consequently, at each fixed
	time, the $L^2$ norms of the $A_j(t)v$ control the corresponding $\Sigma$
	quantities.  The point of \eqref{eq:commutation} is that these same
	quantities also pass through the Duhamel propagator.
	
	Still by \cite[Lemma~4.1]{AntonelliCarlesDrumond}, for the nonlinearity
	\[
	F(u)=|u|^{2\sigma}u,
	\]
	the vector fields satisfy
	\begin{equation}
		|A_j(t)F(u)|\le C_\sigma |u|^{2\sigma}|A_j(t)u|,
		\qquad 0\le j\le4.
		\label{eq:chain-rule}
	\end{equation}
	
	For $v\in\Sigma$ and
	\[
	2\le r\le \frac{2d}{d-2}\quad \text{if } d\ge3,
	\qquad
	2\le r<\infty\quad \text{if } d=2,
	\]
	the weighted Gagliardo--Nirenberg estimate of \cite[Lemma 4.1]{AntonelliCarlesDrumond} reads
	\begin{equation}
		\norm{v}_{L^r}
		\le C_r\langle t\rangle^{-m(1/2-1/r)}
		\norm{v}_2^{1-\delta}
		\left(\sum_{j=1}^4\norm{A_j(t)v}_2\right)^\delta,
		\quad
		\delta=d\left(\frac12-\frac1r\right).
		\label{eq:weighted-GN}
	\end{equation}
	The factor $\langle t\rangle^{-m(1/2-1/r)}$ records that only the
	$m=d-n$ unconfined directions produce large-time decay. The exponent $\delta$ is the standard Gagliardo--Nirenberg interpolation exponent in dimension $d$.
	
	For a fixed $\psi\in\mathcal S(\R^n)$,
	\begin{equation}\label{eq:Lp-estimate}
		\sup_{t\in\R}\norm{e^{-itH_x}\psi}_{L^p(\R^n)}<\infty,
	\end{equation}
	where
	\[
	2\le p\le\infty \quad (n=1),\qquad
	2\le p<\infty \quad (n=2),\qquad
	2\le p\le\frac{2n}{n-2}\quad (n\ge3).
	\]
	Indeed, by \eqref{eq:vector-fields}--\eqref{eq:Heisenberg} and the
	$L^2$-unitarity of $e^{-itH_x}$,
	\[
	\norm{\nabla_x e^{-itH_x}\psi}_{L^2}^2
	+\norm{x e^{-itH_x}\psi}_{L^2}^2
	=
	\norm{\nabla_x\psi}_{L^2}^2+\norm{x\psi}_{L^2}^2.
	\]
	Hence \eqref{eq:Lp-estimate} follows from the Sobolev embedding.
	
	\subsection{Admissible exponents and mixed spaces}
	
	For an integer $D\ge1$, a pair $(a,r)$ is called $D$-admissible if
	\begin{equation}
		2\le a,r\le\infty,
		\qquad
		\frac2a=D\left(\frac12-\frac1r\right),
		\label{eq:admissible}
	\end{equation}
	or equivalently,
	\[
	\frac2a+\frac Dr=\frac D2,
	\]
	with the usual exclusion $(a,r,D)=(2,\infty,2)$.  Choose $(P,Q,R)$ so that
	\begin{equation}
		(Q,R)\text{ is $d$-admissible},
		\qquad
		(P,R)\text{ is $m$-admissible}.
		\label{eq:PQR}
	\end{equation}
	The exponent $Q$ measures time integrability on one bounded interval, while
	$P$ measures summability over all such intervals.  Put
	\[
	I_\gamma=\pi[\gamma-1,\gamma+1),\qquad \gamma\in\mathbb Z,
	\]
	and define
	\begin{equation}
		\norm{v}_{\ell^PL^QL^R(J)}
		=\left(\sum_{\gamma\in\mathbb Z}
		\norm{v}_{L^Q(I_\gamma\cap J;L^R(\R^d))}^P\right)^{1/P},
		\label{eq:mixed-norm}
	\end{equation}
	with the usual supremum interpretation when $P=\infty$.
	
	Choose $(\Th,\Omega,K)$ so that
	\begin{equation}
		\frac1{P'}=\frac1P+\frac{2\sigma}{\Th},
		\qquad
		\frac1{Q'}=\frac1Q+\frac{2\sigma}{\Omega},
		\qquad
		\frac1{R'}=\frac1R+\frac{2\sigma}{K}.
		\label{eq:Holder-relations}
	\end{equation}
	
	For a time interval $J\subset\R$, define
	\begin{align}
		\norm{u}_{X(J)}
		&:=\sum_{j=0}^4\left[
		\norm{A_ju}_{L^\infty(J;L^2)}
		+\norm{A_ju}_{\ell^PL^QL^R(J)}\right],
		\label{eq:X-space}\\
		\norm{G}_{N(J)}
		&:=\sum_{j=0}^4\norm{A_jG}_{\ell^{P'}L^{Q'}L^{R'}(J)},
		\qquad
		\norm{G}_{N_0(J)}:=\norm{G}_{\ell^{P'}L^{Q'}L^{R'}(J)},
		\label{eq:N-spaces}\\
		\norm{u}_{Y(J)}&:=\norm{u}_{\ell^\Th L^\Omega L^K(J)},
		\label{eq:Y-space}\\
		d_J(u,v)&:=\norm{u-v}_{\ell^PL^QL^R(J)}.
		\label{eq:weak-metric}
	\end{align}
	We will see in later sections that, the solution ball will be bounded in the strong $X$ norm, whereas the
	fixed-point contraction will use the weaker metric $d_J$.
	
	\subsection{The embedding of \texorpdfstring{$X$ into $Y$}{X into Y}}
	
	Set
	\begin{equation}
		a=m\left(\frac12-\frac1K\right).
		\label{eq:a-delta}
	\end{equation}
	Applying \eqref{eq:weighted-GN} with $r=K$ and using the
	$L^\infty_tL^2$ part of $X(J)$ gives
	\begin{equation}
		\norm{u(t)}_{L^K}
		\le C\langle t\rangle^{-a}\norm{u}_{X(J)}.
		\label{eq:pointwise-decay}
	\end{equation}
	We now integrate this pointwise-in-time estimate over one interval.  Since $I_\gamma$ has length $2\pi$,
	\begin{align}
		\norm{u}_{L^\Omega(I_\gamma\cap J;L^K)}
		&\le C\norm{\langle t\rangle^{-a}}_{L^\Omega(I_\gamma\cap J)}
		\norm{u}_{X(J)}\notag\\
		&\le C(2\pi)^{1/\Omega}
		\sup_{t\in I_\gamma}\langle t\rangle^{-a}\norm{u}_{X(J)}\notag\\
		&\le C\langle\gamma\rangle^{-a}\norm{u}_{X(J)}.
		\label{eq:local-Y}
	\end{align}
	Taking the $\ell^\Th$ norm, we obtain
	\begin{align*}
		\norm{u}_{Y(J)}^\Theta
		&=\sum_{\gamma\in\mathbb Z}
		\norm{u}_{L^\Omega(I_\gamma\cap J;L^K)}^\Theta\\
		&\le C\norm{u}_{X(J)}^\Theta
		\sum_{\gamma\in\mathbb Z}\langle\gamma\rangle^{-a\Theta}.
	\end{align*}
	The last numerical series converges exactly when
	\begin{equation}
		a\Th=m\left(\frac12-\frac1K\right)\Th>1.
		\label{eq:ordinary-summability}
	\end{equation}
	
	This shows that under the condition \eqref{eq:ordinary-summability}, we have the embedding
	\begin{equation}
		\norm{u}_{Y(J)}\le C\norm{u}_{X(J)}
		\label{eq:Y-embedding}.
	\end{equation}
	This embedding will be applied directly in Section~3.  Section~4
	requires a strengthened version with an additional time weight.
	
	Concrete choices of the exponents are recorded below (see \cite[Section 4]{AntonelliCarlesDrumond}).
	
	When $d=2$ and $n=m=1$, one can take $R=2+\eta$ with
	$\eta>0$ small and
	\begin{equation}
		K=\frac{2\sigma(2+\eta)}{\eta},\quad
		P=\frac{4(2+\eta)}{\eta},\quad
		Q=\frac{2(2+\eta)}{\eta},\quad
		\Th=\frac{2\sigma(4+2\eta)}{4+\eta},\quad
		\Omega=(2+\eta)\sigma.
		\label{eq:d2-exponents}
	\end{equation}
	For $d\ge3$, a convenient choice is
	\begin{align}
		K&=\frac{2d}{d-2},
		&R&=\frac{2d}{d-(d-2)\sigma},
		&P&=\frac{4d}{m\sigma(d-2)},
		&Q&=\frac4{\sigma(d-2)},\notag\\
		\Th&=\frac{4d\sigma}{2d-m(d-2)\sigma},
		&\Omega&=\frac{4\sigma}{2-(d-2)\sigma}.
		\label{eq:high-d-exponents}
	\end{align}
	With these exponents, \eqref{eq:ordinary-summability} is equivalent to
	\begin{equation}
		\sigma>\frac{2d}{(d+2)m}.
		\label{eq:ordinary-lower}
	\end{equation}
	
	\subsection{Strichartz estimates}
	
	The global Strichartz estimate under partial harmonic confinement   (\cite[Theorem~3.4, (3.4)]{AntonelliCarlesDrumond}), together
	with \eqref{eq:commutation} and the $L^2$-unitarity of $U(\cdot)$, gives 
	\begin{equation}
		\norm{U(\cdot)f}_{X(\R)}\le C\norm f_\Sigma.
		\label{eq:hom-Strichartz}
	\end{equation}
	
	The retarded inhomogeneous estimates that will be used later are
	\begin{align}
		\norm{\int_{-\infty}^tU(t-s)G(s)\dd s}_{X(\R)}
		&\le C\norm G_{N(\R)},
		\label{eq:strong-retarded}\\
		\norm{\int_{-\infty}^tU(t-s)G(s)\dd s}_{\ell^PL^QL^R(\R)}
		&\le C\norm G_{N_0(\R)}.
		\label{eq:weak-retarded}
	\end{align}
	We remark that although \cite[Theorem~3.4, (3.6)]{AntonelliCarlesDrumond} states with $\int_0^t$, integrating from $-\infty$ instead is essentially not a different Strichartz estimate. It is a consequence of the standard truncation and completeness argument (see for instance, \cite[Corollary 2.3.6]{Cazenave}).
	
	We will also use the dual homogeneous estimate on the negative half-line:
	\begin{equation}
		\norm{\int_{-\infty}^0U(-t)G(t)\dd t}_{L^2}
		\le C\norm G_{N_0(( -\infty,0])},
		\label{eq:dual-negative}
	\end{equation}
	which follows from \cite[Theorem~3.4, (3.5)]{AntonelliCarlesDrumond} regarding $G$ as its zero extension to $\R$.

	\section{The coefficient depends on the confined variable:
		\texorpdfstring{$\la=\la(x)$}{lambda = lambda(x)}}
	
	Throughout this section, assume
	\begin{equation}
		\la\in W^{1,\infty}(\R^n;\R),
		\qquad
		\frac{2d}{(d+2)m}<\sigma,
		\qquad
		\sigma<\frac2{d-2}\quad(d\ge3).
		\label{eq:x-assumptions}
	\end{equation}
	For $d=2$, $n=m=1$, the lower bound is simply $\sigma>1$.
	
	\subsection{Coefficient commutators and nonlinear estimates}
	
	Write $\cG(u)=\la(x)F(u)$.  Since $A_3,A_4$ act only in $y$,
	\begin{equation}
		A_j\cG(u)=\la(x)A_jF(u),\qquad j=3,4.
		\label{eq:x-free-fields}
	\end{equation}
	For the confined fields, the product rule gives
	\begin{align}
		A_1\cG(u)
		&=\la A_1F(u)-i\cos t\,(\nabla_x\la)F(u),
		\label{eq:x-A1}\\
		A_2\cG(u)
		&=\la A_2F(u)+i\sin t\,(\nabla_x\la)F(u).
		\label{eq:x-A2}
	\end{align}
	Compared with the constant $\lambda$ case studied in \cite{AntonelliCarlesDrumond}, the new factors $\sin t$ and $\cos t$ are harmless since they are uniformly bounded.
	
	Fix $B\in\{A_0,A_1,A_2,A_3,A_4\}$.  By \eqref{eq:chain-rule}, on a fixed
	time interval $I_\gamma$,
	\[
	\abs{BF(u)}\le C|u|^{2\sigma}|Bu|.
	\]
	First apply H\"older in space, using the third relation in \eqref{eq:Holder-relations}:
	\begin{align}
		\norm{BF(u)(t)}_{L^{R'}}
		&\le C\norm{|u(t)|^{2\sigma}}_{L^{K/(2\sigma)}}
		\norm{Bu(t)}_{L^R}\notag\\
		&=C\norm{u(t)}_{L^K}^{2\sigma}\norm{Bu(t)}_{L^R} \label{eq:space-Holder}.
	\end{align}
	Next apply H\"older in time on $I_\gamma\cap J$, using the second relation:
	\begin{align}
		\norm{BF(u)}_{L^{Q'}(I_\gamma\cap J;L^{R'})}
		\le C
		\norm{u}_{L^\Omega(I_\gamma\cap J;L^K)}^{2\sigma}
		\norm{Bu}_{L^Q(I_\gamma\cap J;L^R)} \label{eq:time-Holder}.
	\end{align}
	Finally set
	\[
	a_\gamma=\norm{u}_{L^\Omega(I_\gamma\cap J;L^K)},
	\qquad
	b_\gamma=\norm{Bu}_{L^Q(I_\gamma\cap J;L^R)}.
	\]
	Applying H\"older for sequences, using the first relation in
	\eqref{eq:Holder-relations}, gives
	\begin{align}
		\norm{BF(u)}_{\ell^{P'}L^{Q'}L^{R'}(J)}
		&\le C\norm{a_\gamma^{2\sigma}b_\gamma}_{\ell^{P'}}\notag\\
		&\le C\norm{a}_{\ell^\Theta}^{2\sigma}
		\norm{b}_{\ell^P}\notag\\
		&=C\norm{u}_{Y(J)}^{2\sigma}
		\norm{Bu}_{\ell^P L^Q L^R(J)}.\label{eq:sequence-Holder}
	\end{align}
	
	We remark that the inequality above was claimed in the proof of Theorem 1.5 in \cite[Section 4]{AntonelliCarlesDrumond}, although there are no detailed derivations at that place. We add the details here for self-completeness.
	
	Sum over $B$ and use \eqref{eq:Y-embedding} to obtain
	\[\norm{F(u)}_{N(J)}
	\le C\norm{u}_{X(J)}^{2\sigma+1}.\]

	Also note that, for any $z,w\in\mathbb C$ and $\sigma>0$,
	\begin{equation}\label{eq:point-lip}
		\abs{|z|^{2\sigma}z-|w|^{2\sigma}w}
		\le C_\sigma(|z|^{2\sigma}+|w|^{2\sigma})|z-w|.
	\end{equation}
	Applying exactly the same three H\"older inequalities, now to
	\eqref{eq:point-lip}, yields
	\[
	\norm{F(u)-F(v)}_{N_0(J)}
	\le C\bigl(\norm{u}_{X(J)}^{2\sigma}
	+\norm{v}_{X(J)}^{2\sigma}\bigr)d_J(u,v).
	\]
	For example, the contribution containing $u$ is
	\begin{align*}
		&\norm{|u|^{2\sigma}(u-v)}_{\ell^{P'}L^{Q'}L^{R'}(J)}\\
		&\qquad\le
		\norm{u}_{\ell^\Theta L^\Omega L^K(J)}^{2\sigma}
		\norm{u-v}_{\ell^P L^Q L^R(J)}
		\le C\norm{u}_{X(J)}^{2\sigma}d_J(u,v),
	\end{align*}
	and the term containing $v$ is identical.
	
	Combining \eqref{eq:x-A1}--\eqref{eq:x-A2} with the estimates above for $F$, we obtain 
	\begin{equation}
		\norm{\cG(u)}_{N(J)}
		\le C\norm\la_{W^{1,\infty}}
		\norm u_{X(J)}^{2\sigma+1},
		\label{eq:x-strong-nonlinear}
	\end{equation}
	and 
	\begin{equation}
		\norm{\cG(u)-\cG(v)}_{N_0(J)}
		\le C\norm\la_\infty
		\bigl(\norm u_{X(J)}^{2\sigma}+\norm v_{X(J)}^{2\sigma}\bigr)
		d_J(u,v).
		\label{eq:x-weak-difference}
	\end{equation}
	
	Notice that the left-hand side of the weak estimate
	\eqref{eq:x-weak-difference} is $N_0$, not the differentiated space $N$.  This is
	deliberate: it will be used later for contraction in $d_J$.

	\subsection{Fixed point and incoming wave operator}
	
	\begin{proposition}\label{prop:xforward}
		Under the assumptions of this section, for sufficiently small $f\in\Sigma$, the equation
		\begin{equation}\label{eq:incoming-duhamel}
			u(t)=U(t)f-i\int_{-\infty}^{t}
			U(t-s)\cG(u(s))\dd s.
		\end{equation}
		has a unique small solution in $X(\R)$. It satisfies \begin{equation}\label{eq:incoming}
			\norm{U(-t)u(t)-f}_{\Sigma}\longrightarrow0\quad(t\to-\infty),
		\end{equation}
		and
		\begin{equation}
			\Om^-_{\la(x),\sigma}(f)
			=f-i\int_{-\infty}^0U(-t)
			[\la(x)|u(t)|^{2\sigma}u(t)]\dd t.
			\label{eq:x-wave-formula}
		\end{equation}
		Moreover, for $f=\varepsilon\varphi$ with fixed $\varphi\in\Sigma$ and sufficiently small $\varepsilon> 0$, the corresponding solution $u_\varepsilon$ of \eqref{eq:incoming-duhamel} satisfies
		\begin{align}
			\norm{u_\varepsilon}_X&\le C_\varphi\varepsilon,
			\label{eq:x-small-solution}\\
			d_\R(u_\varepsilon,\varepsilon U(\cdot)\varphi)
			&\le C_{\varphi,\la,\sigma}\varepsilon^{2\sigma+1}.
			\label{eq:x-linear-approx}
		\end{align}
	\end{proposition}
	\begin{proof}
		For prescribed incoming state $f\in\Sigma$, define
		\begin{equation}
			\Phi_f(u)(t)
			=U(t)f-i\int_{-\infty}^tU(t-s)\cG(u(s))\dd s.
			\label{eq:x-fixed-map}
		\end{equation}
		Equations \eqref{eq:hom-Strichartz}, \eqref{eq:strong-retarded}, and
		\eqref{eq:x-strong-nonlinear} imply
		\begin{equation}
			\norm{\Phi_f(u)}_X
			\le C_0\norm f_\Sigma
			+C_1\norm\la_{W^{1,\infty}}\norm u_X^{2\sigma+1}.
			\label{eq:x-map-bound}
		\end{equation}
		Similarly, \eqref{eq:weak-retarded} and \eqref{eq:x-weak-difference} imply
		\begin{equation}
			d_\R(\Phi_f(u),\Phi_f(v))
			\le C_2\norm\la_\infty
			(\norm u_X^{2\sigma}+\norm v_X^{2\sigma})d_\R(u,v).
			\label{eq:x-contraction}
		\end{equation}
		
		Put $M=2C_0\norm f_\Sigma$ and
		\[
		B_M=\{u\in X(\R):\norm u_X\le M\}.
		\]
		If $\norm f_\Sigma$ is small enough, then
		\[
		C_1\norm\la_{W^{1,\infty}}M^{2\sigma}\le\frac12,
		\qquad
		2C_2\norm\la_\infty M^{2\sigma}<1.
		\]
		Therefore $\Phi_f$ maps $B_M$ into itself and is a contraction for $d_\R$.
		
		For completeness, the ball $B_M$ is closed in this weaker metric following Kato’s standard argument. Indeed, if $\{u_k\}$ is $d_\R$-Cauchy, then it converges
		strongly in $\ell^PL^QL^R$.  Its uniform $X$ bound gives weak or weak-star
		convergent subsequences for every $A_ju_k$ in the component spaces of $X$.
		The distributional closedness of $A_j$ identifies these limits with
		$A_ju$, and the lower semicontinuity gives 
		\[\norm{u}_{X}\le\liminf_{n\to\infty}\norm{u_n}_{X}\leq M.\]
		Hence the
		fixed point $u\in B_M$ exists and is unique in the small ball.
		
		Next, we verify the incoming asymptotic condition. From
		\eqref{eq:incoming-duhamel},
		\[
		U(-t)u(t)-f
		=
		-i\int_{-\infty}^{t}U(-s)\cG(u(s))\,\dd s.
		\]
		Using
		\[
		A_j(0)U(-s)=U(-s)A_j(s),
		\]
		together with the dual Strichartz estimate, we obtain, for
		$j=0,\ldots,4$,
		\[
		\norm{A_j(0)(U(-t)u(t)-f)}_{L^2}
		\le
		C\norm{A_j\cG(u)}
		{\ell^{P'}L^{Q'}L^{R'}((-\infty,t])}.
		\]
		Since $\cG(u)\in N(\R)$, for each $j=0,\ldots,4$,
		\[
		\norm{A_j\cG(u)}
		{\ell^{P'}L^{Q'}L^{R'}(\R)}<\infty.
		\]
		By the definition of the $\ell^{P'}$ norm and the fact that $P'<\infty$,
		the far-left tail of this summable sequence tends to zero. Hence
		\[
		\norm{A_j\cG(u)}
		{\ell^{P'}L^{Q'}L^{R'}((-\infty,t])}
		\longrightarrow0
		\qquad (t\to-\infty).
		\]
		Therefore
		\[
		\norm{A_j(0)(U(-t)u(t)-f)}_{L^2}
		\longrightarrow0
		\qquad (t\to-\infty).
		\]
		Recall that
		\[
		A_0(0)=\Id,\qquad
		A_1(0)=-i\nabla_x,\qquad
		A_2(0)=x,\qquad
		A_3(0)=-i\nabla_y,\qquad
		A_4(0)=y,
		\]
		we conclude that
		\[
		\norm{U(-t)u(t)-f}_{\Sigma}
		\longrightarrow0
		\qquad (t\to-\infty).
		\]
		
		Evaluating \eqref{eq:incoming-duhamel} at zero proves \eqref{eq:x-wave-formula}.
		
		Finally, \eqref{eq:x-small-solution} directly follows from \eqref{eq:x-map-bound} for sufficiently small $\varepsilon> 0$. Then \eqref{eq:x-linear-approx} follows from \eqref{eq:x-strong-nonlinear}:
		\begin{align*}
			d_\R(u_\eps,\eps U(\cdot)\varphi)
			&\le C\norm{\cG(u_\eps)}_{N_0}\\
			&\le C\norm{\cG(u_\eps)}_{N}
			\le C\norm\la_{W^{1,\infty}}\norm{u_\eps}_{X}^{2\sigma+1}
			\le C\eps^{2\sigma+1}.
		\end{align*}
	\end{proof}
	
	\subsection{Recovery of \texorpdfstring{$\la(x)$}{lambda(x)}}
	We adopt the complex inner product convention
	$\ip fg=\int_{\R^d}f\overline{g}$.
	For $\varphi\in \Sigma$, define
	\begin{equation}
		\cC^-_{\la(x),\sigma}(\varphi)
		=\int_{-\infty}^0\int_{\R^n\times\R^m}
		\la(x)|U(t)\varphi(x,y)|^{2\sigma+2}\dd x\dd y\dd t.
		\label{eq:x-leading-functional}
	\end{equation}
	The integral is absolutely convergent because the $(2\sigma+2)$ power of
	\eqref{eq:weighted-GN} (with $r= 2\sigma+2$) decays like $\langle t\rangle^{-m\sigma}$, and
	\eqref{eq:x-assumptions} implies $m\sigma>1$.
	
	\begin{proposition}
		For fixed $\varphi\in\Sigma$, as $\varepsilon\downarrow0$,
		\begin{align}
			\Om^-_{\la(x),\sigma}(\varepsilon\varphi)
			={}&\varepsilon\varphi
			-i\varepsilon^{2\sigma+1}\int_{-\infty}^0U(-t)
			[\la(x)|U(t)\varphi|^{2\sigma}U(t)\varphi]\dd t
			+R_\varepsilon,
			\label{eq:x-Born}
		\end{align}
		where
		\begin{equation}
			\norm{R_\varepsilon}_2
			\le C_{\varphi,\la,\sigma}\varepsilon^{4\sigma+1}.
			\label{eq:x-remainder}
		\end{equation}
		Consequently, for
		\begin{equation}
			B_\varphi(\varepsilon)
			:=i\ip{(\Om^-_{\la(x),\sigma}-\Id)(\varepsilon\varphi)}{\varphi},
			\label{eq:x-B-def}
		\end{equation}
		one has
		\begin{equation}
			B_\varphi(\varepsilon)
			=\varepsilon^{2\sigma+1}\cC^-_{\la(x),\sigma}(\varphi)
			+O(\varepsilon^{4\sigma+1}).
			\label{eq:x-B-expansion}
		\end{equation}
	\end{proposition}
	
	\begin{proof}
		By \eqref{eq:x-wave-formula},
		\[
		R_\varepsilon=-i\int_{-\infty}^0U(-t)
		[\cG(u_\varepsilon)-\cG(\varepsilon U(t)\varphi)]\dd t.
		\]
		By \eqref{eq:dual-negative}, \eqref{eq:x-weak-difference},
		\eqref{eq:x-small-solution}, and \eqref{eq:x-linear-approx},
		\begin{align*}
			\norm{R_\varepsilon}_2
			&\le C\norm\la_\infty
			(\norm{u_\varepsilon}_X^{2\sigma}
			+\norm{\varepsilon U(\cdot)\varphi}_X^{2\sigma})
			d_\R(u_\varepsilon,\varepsilon U(\cdot)\varphi)\\
			&\le C_{\varphi,\la,\sigma}\varepsilon^{4\sigma+1}.
		\end{align*}
		Pairing the leading term with $\varphi$ and using the unitarity of $U(t)$ proves
		\eqref{eq:x-B-expansion}.
	\end{proof}
	
	For the recovery of $\sigma$, note that whenever $\cC^-_{\la(x),\sigma}(\varphi)\ne0$,
	\begin{equation}
		L_\varphi:=\lim_{\varepsilon\downarrow0}
		\frac{B_\varphi(2\varepsilon)}{B_\varphi(\varepsilon)}
		=2^{2\sigma+1},
		\qquad
		\sigma=\frac12\left(\log_2L_\varphi-1\right).
		\label{eq:x-sigma-recovery}
	\end{equation}
	This can be viewed as an analogue of the amplitude-ratio argument of Murphy \cite[Corollary~4.2]{Murphy}. 
	
	Once $\sigma$ is determined, the next goal is the recovery of $\lambda(x)$.
	
	For convenience, we put $p=2\sigma+2$. We fix $0\ne g\in\mathcal S(\R^m)$. A specific choice could be the Gaussian (see for instance, \cite[Section 4]{Murphy}). For arbitrary
	$\psi\in\mathcal S(\R^n)$ and the parameter $h> 0$, set
	\begin{equation}
		g_h(y)=h^{-m/2}g(y/h),
		\qquad
		\varphi_{h,\psi}(x,y)=\psi(x)g_h(y).
		\label{eq:x-probe}
	\end{equation}
	By the scaling property of the free propagator (see for instance, \cite[Remark 2.2.2]{Cazenave}),
	\[(e^{it\Delta_y/2}g_h)(y)= h^{-m/2}(e^{i\frac{t}{h^2}\Delta_y/2}g)(\frac{y}{h}).\]
	With the substitution $t=h^2s$ in \eqref{eq:x-leading-functional}, we obtain
	\begin{align}
		h^{m\sigma-2}\cC^-_{\la(x),\sigma}(\varphi_{h,\psi})
		=\int_{-\infty}^0
		&\norm{e^{is\Delta_y/2}g}_{L^p_y}^p\notag\\[-1mm]
		&\times\left[\int_{\R^n}\la(x)
		|e^{-ih^2sH_x}\psi(x)|^p\dd x\right]\dd s.
		\label{eq:x-scaled-functional}
	\end{align}
	Note that
	\[
	\norm{e^{is\Delta_y/2}g}_{L^p}^p\le C_g\langle s\rangle^{-m\sigma}
	\]
	and $m\sigma>1$. Moreover, for each fixed Schwartz function $\psi$, by
	\eqref{eq:Lp-estimate},
	\[
	\sup_{t\in\R}\norm{e^{-itH_x}\psi}_{L^p}<\infty.
	\]
	Therefore, dominated convergence theorem yields
	\begin{equation}
		\lim_{h\downarrow0}h^{m\sigma-2}
		\cC^-_{\la(x),\sigma}(\varphi_{h,\psi})
		=D^-_{\sigma,g}\int_{\R^n}\la(x)|\psi(x)|^p\dd x,
		\label{eq:x-functional-limit}
	\end{equation}
	where
	\begin{equation}
		D^-_{\sigma,g}
		:=\int_{-\infty}^0\norm{e^{is\Delta_y/2}g}_{L^p}^p\dd s
		\in(0,\infty).
		\label{eq:D-def}
	\end{equation}
	Thus the incoming wave operator determines the nonlinear testing functional
	\[
	Q_\la(\psi)=\int_{\R^n}\la(x)|\psi(x)|^p\dd x.
	\]
	
	Now pointwise reconstruction of $\lambda$ from $Q_\la$ is based on the standard localization argument.
	
	Indeed, for a fixed $x_0\in \R^n$, choose $0\ne\eta\in C_c^\infty(\R^n)$ and define
	\begin{equation}
		\psi_{\rho,x_0}(x)=\rho^{-n/p}
		\eta\left(\frac{x-x_0}{\rho}\right).
		\label{eq:x-localizer}
	\end{equation}
	Since $\lambda\in W^{1,\infty}(\R^n)$, we may take its Lipschitz
	continuous representative. Hence, for every $x_0\in\R^n$,
	\[
	\lambda(x_0+\rho z)\to\lambda(x_0)
	\qquad\text{as }\rho\downarrow0,
	\]
	and
	\[
	|\lambda(x_0+\rho z)|\,|\eta(z)|^p
	\le \|\lambda\|_{L^\infty}|\eta(z)|^p.
	\]
	Therefore, by dominated convergence theorem,
	\[
	Q_\lambda(\psi_{\rho,x_0})=\int_{\R^n}\la(x_0+\rho z)|\eta(z)|^p\dd z
	\longrightarrow
	\lambda(x_0)\|\eta\|_{L^p}^p
	\qquad\text{as }\rho\downarrow0.
	\]
	
	This proves the first half of Theorem \ref{thm:main}.

	\section{The coefficient depends on the free variable:
		\texorpdfstring{$\la=\la(y)$}{lambda = lambda(y)}}
	
	Throughout this section, assume
	\begin{equation}
		\la\in W^{1,\infty}(\R^m;\R),
		\qquad
		\frac{4d}{(d+2)m}<\sigma,
		\qquad
		\sigma<\frac2{d-2}\quad(d\ge3).
		\label{eq:y-assumptions}
	\end{equation}
	The stronger lower bound is the price of controlling the factor
	$t\nabla_y\la$.
	
	\subsection{The new commutator term and strengthened embedding}
	
	Write $\cG(u)=\la(y)F(u)$.  The coefficient commutes with
	$A_0,A_1,A_2$, while
	\begin{align}
		A_3\cG(u)&=\la A_3F(u)-i(\nabla_y\la)F(u),
		\label{eq:y-A3}\\
		A_4\cG(u)&=\la A_4F(u)+it(\nabla_y\la)F(u).
		\label{eq:y-A4}
	\end{align}
	The $A_3$ error has no time growth.  The last term in \eqref{eq:y-A4} is the
	genuinely new difficulty.
	
	We define the weighted auxiliary norm
	\begin{equation}
		\norm u_{Y_1(J)}
		:=\norm{\langle t\rangle^{1/(2\sigma)}u}
		_{\ell^\Th L^\Omega L^K(J)},
		\label{eq:Y1-space}
	\end{equation}
	which is motivated by the inequality 
	\begin{equation}
		|t||u|^{2\sigma+1}
		\le\bigl(\langle t\rangle^{1/(2\sigma)}|u|\bigr)^{2\sigma}|u|.\label{eq:tu-estimate}
	\end{equation}
	
	Since
	$\langle t\rangle^{1/(2\sigma)}\ge1$, it is clear that 
	\[\|u\|_{Y(J)}\le\|u\|_{Y_1(J)}.\]
	
	Multiplying \eqref{eq:pointwise-decay} by
	$\langle t\rangle^{1/(2\sigma)}$ and repeating the derivation of
	\eqref{eq:local-Y} gives
	\begin{equation}
		\norm{\langle t\rangle^{1/(2\sigma)}u}
		_{L^\Omega(I_\gamma\cap J;L^K)}
		\le C\langle\gamma\rangle^{-a+1/(2\sigma)}\norm u_{X(J)}.
		\label{eq:Y1-local}
	\end{equation}
	Taking the $\ell^\Th$ norm gives
	\[
	\|u\|_{Y_1(J)}^\Th
	\le C\|u\|_{X(J)}^\Th
	\sum_{\gamma\in\mathbb Z}
	\langle\gamma\rangle^{-\Th(a-1/(2\sigma))}.
	\]
	The numerical series converges provided
	\begin{equation}
		\left[m\left(\frac12-\frac1K\right)-\frac1{2\sigma}\right]\Th>1.
		\label{eq:strong-summability}
	\end{equation}
	Consequently,
	\begin{equation} 
		\norm u_{Y_1(J)}\le C\norm u_{X(J)}
		\label{eq:Y1-embedding}
	\end{equation}
	under \eqref{eq:strong-summability}.  
	
	For $d\ge3$, substitution of \eqref{eq:high-d-exponents} gives
	\begin{equation}
		\Th\left(\frac md-\frac1{2\sigma}\right)
		=\frac{4m\sigma-2d}{2d-m(d-2)\sigma}.
	\end{equation}
	Condition \eqref{eq:strong-summability} is therefore equivalent to
	$$m(d+2)\sigma>4d.$$  
	
	For $d=2$, substitution of
	\eqref{eq:d2-exponents} and then taking $\eta>0$ sufficiently small gives
	$\sigma>2$. This motivates the lower bound in \eqref{eq:y-assumptions}. More specifically, in the cases
	\[
	(d,n,m)=(2,1,1),\ (3,1,2),\ (4,1,3),
	\]
	respective ranges are $\sigma>2$, $6/5<\sigma<2$, and
	$8/9<\sigma<1$.
	
	\subsection{Nonlinear estimates and forward problem}
	
	The ordinary terms are handled by
	\eqref{eq:space-Holder}--\eqref{eq:sequence-Holder}.  From
	\eqref{eq:y-A3},
	\begin{equation}
		\norm{(\nabla_y\la)|u|^{2\sigma}u}_{\ell^{P'}L^{Q'}L^{R'}(J)}
		\le\norm{\nabla_y\la}_\infty
		\norm u_{Y(J)}^{2\sigma}
		\norm u_{\ell^PL^QL^R(J)}.
		\label{eq:y-A3-estimate}
	\end{equation}
	For the difficult term in \eqref{eq:y-A4}, \eqref{eq:tu-estimate} and the definition of $Y_1$ give
	\begin{align}
		&\norm{t(\nabla_y\la)|u|^{2\sigma}u}
		_{\ell^{P'}L^{Q'}L^{R'}(J)}\notag\\
		&\qquad\le\norm{\nabla_y\la}_\infty
		\norm u_{Y_1(J)}^{2\sigma}
		\norm u_{\ell^PL^QL^R(J)}.
		\label{eq:y-A4-estimate}
	\end{align}
	Using \eqref{eq:Y1-embedding} (implying \eqref{eq:Y-embedding}) yields
	\begin{equation}
		\norm{\cG(u)}_{N(J)}
		\le C\norm\la_{W^{1,\infty}
			\norm u_{X(J)}^{2\sigma+1}.}
		\label{eq:y-strong-nonlinear}
	\end{equation}
	Without applying any vector field, the same difference argument as in the Subsection 3.1 gives
	\begin{equation}
		\norm{\cG(u)-\cG(v)}_{N_0(J)}
		\le C\norm\la_\infty
		(\norm u_{X(J)}^{2\sigma}+\norm v_{X(J)}^{2\sigma})d_J(u,v).
		\label{eq:y-weak-difference}
	\end{equation}
	
	The two nonlinear estimates above are exactly the same as their counterparts in the $x$-dependent case. Using the same Kato's fixed point argument, we obtain the following analogue of Proposition \ref{prop:xforward}.
	
	\begin{proposition}\label{prop:yforward}
		Under the assumptions of this section, for sufficiently small $f\in\Sigma$, the nonlinear equation
		has a unique small solution $u$ in $X(\R)$ satisfying the incoming condition and
		\begin{equation}
			\Om^-_{\la(y),\sigma}(f)
			=f-i\int_{-\infty}^0U(-t)
			[\la(y)|u(t)|^{2\sigma}u(t)]\dd t.
			\label{eq:y-wave-formula}
		\end{equation}
		Moreover, for $f=\varepsilon\varphi$ with fixed $\varphi\in\Sigma$ and sufficiently small $\varepsilon> 0$, the corresponding solution $u_\varepsilon$ satisfies
		\begin{equation}
			\norm{u_\varepsilon}_X\le C_\varphi\varepsilon,
			\qquad
			d_\R(u_\varepsilon,\varepsilon U(\cdot)\varphi)
			\le C_{\varphi,\la,\sigma}\varepsilon^{2\sigma+1}.
			\label{eq:y-small-bounds}
		\end{equation}
	\end{proposition}
	
	\subsection{Recovery of \texorpdfstring{$\la(y)$}{lambda(y)}}
	
	Define
	\begin{equation}
		\cC^-_{\la(y),\sigma}(\varphi)
		=\int_{-\infty}^0\int_{\R^n\times\R^m}
		\la(y)|U(t)\varphi(x,y)|^{2\sigma+2}\dd x\dd y\dd t.
		\label{eq:y-leading-functional}
	\end{equation}
	Exactly as in the proof of the $x$-dependent case,
	\begin{align}
		\Om^-_{\la(y),\sigma}(\varepsilon\varphi)
		={}&\varepsilon\varphi
		-i\varepsilon^{2\sigma+1}\int_{-\infty}^0U(-t)
		[\la(y)|U(t)\varphi|^{2\sigma}U(t)\varphi]\dd t
		+R_\varepsilon,
		\label{eq:y-Born}\\
		\norm{R_\varepsilon}_2
		&\le C_{\varphi,\la,\sigma}\varepsilon^{4\sigma+1}.
		\label{eq:y-remainder}
	\end{align}
	Therefore, with
	\begin{equation}
		B_\varphi(\varepsilon)
		:=i\ip{(\Om^-_{\la(y),\sigma}-\Id)(\varepsilon\varphi)}{\varphi},
		\label{eq:y-B-def}
	\end{equation}
	we have
	\begin{equation}
		B_\varphi(\varepsilon)
		=\varepsilon^{2\sigma+1}\cC^-_{\la(y),\sigma}(\varphi)
		+O(\varepsilon^{4\sigma+1}).
		\label{eq:y-B-expansion}
	\end{equation}
	
	Hence, $\sigma$ can be determined in the same way as before. Compared with the recovery of $\lambda(x)$ in the previous section, we will see that the recovery of $\lambda(y)$ is more straightforward.
	
	Indeed, we can fix a nonzero Hermite eigenfunction $\psi\in\mathcal S(\R^n)$,
	and $0\ne g\in\mathcal S(\R^m)$.  
	Then
	\begin{equation}
		H_x\psi=E\psi, \label{Hermite}
	\end{equation}
	where $E$ is the corresponding eigenvalue.
	For $y_0\in\R^m$
	and $h>0$, let
	\begin{equation}
		g_{h,y_0}(y)=h^{-m/2}g\left(\frac{y-y_0}{h}\right),
		\qquad
		\varphi_{h,y_0}(x,y)=\psi(x)g_{h,y_0}(y).
		\label{eq:y-probe}
	\end{equation}
	Since 
	$$e^{-itH_x}\psi=e^{-itE}\psi,$$
	free scaling, followed by the substitutions $t=h^2s$ and
	$y=y_0+hz$ in \eqref{eq:y-leading-functional}, gives the identity
	\begin{align}
		h^{m\sigma-2}\cC^-_{\la(y),\sigma}(\varphi_{h,y_0})
		={}&\norm\psi_{L^p(\R^n)}^p
		\int_{-\infty}^0\int_{\R^m}
		\la(y_0+hz)|e^{is\Delta_y/2}g(z)|^p\dd z\dd s,
		\label{eq:y-scaled-functional}
	\end{align}
	where $p=2\sigma+2$. Dominated convergence yields
	\begin{equation}
		\lim_{h\downarrow0}h^{m\sigma-2}
		\cC^-_{\la(y),\sigma}(\varphi_{h,y_0})
		=\la(y_0)\norm\psi_{L^p}^pD^-_{\sigma,g}.
		\label{eq:y-functional-limit}
	\end{equation}
	
	This proves the second half of Theorem \ref{thm:main}.
	
	\section{Further remarks and extensions}
	
	\subsection*{Comparison of the two coefficients}
	To summarize, we make the following comparison between the two cases.
	
	\begin{center}
		\begin{tabular}{@{}p{0.25\textwidth}p{0.32\textwidth}p{0.32\textwidth}@{}}
			\toprule
			Feature & $\la=\la(x)$ & $\la=\la(y)$\\
			\midrule
			Additional differentiated terms
			& $(\sin t\text{ or }\cos t)(\nabla_x\la)F(u)$
			& $(\nabla_y\la)F(u)$ and $t(\nabla_y\la)F(u)$\\[1mm]
			Auxiliary embedding
			& $X\hookrightarrow Y$
			& $X\hookrightarrow Y_1\hookrightarrow Y$\\[1mm]
			Lower power bound
			& $\sigma>2d/((d+2)m)$
			& $\sigma>4d/((d+2)m)$\\[1mm]
			Recovery of $\lambda$
			& concentrate $y$ to freeze time, then localize $x$
			& concentrate $y$ directly at $y_0$\\
			
			\bottomrule
		\end{tabular}
	\end{center}
	
	\subsection*{Outgoing and scattering operators}
	We remark that a parallel theory can be developed for the outgoing wave operator and the scattering operator.
	Indeed, the estimates in earlier sections hold on positive time intervals. For small $f\in\Sigma$, the outgoing solution $u$ satisfies
	\begin{equation}\label{eq:outgoing}
		\begin{aligned}
			u(t)&=U(t)f+i\int_t^\infty U(t-s)\lambda|u(s)|^{2\sigma}u(s)\,ds,
		\end{aligned}
	\end{equation}
	and $\norm{U(-t)u(t)-f}_\Sigma\to0$ as $t\to+\infty$. Evaluating \eqref{eq:outgoing} at zero gives
	\begin{equation}\label{eq:outgoing-wave}
		\Omega^+_{\lambda,\sigma}(f)=u(0)
		=f+i\int_0^\infty U(-s)\lambda|u(s)|^{2\sigma}u(s)\,ds.
	\end{equation}
	In addition, the incoming solution in Proposition \ref{prop:xforward} or \ref{prop:yforward} has an outgoing asymptotic state. For small $f\in\Sigma$, the scattering map $$S_{\lambda,\sigma}=(\Omega^+_{\lambda,\sigma})^{-1}\circ\Omega^-_{\lambda,\sigma}$$ is therefore given by
	\begin{equation}\label{eq:scattering}
		S_{\lambda,\sigma}(f)=f-i\int_{\R}U(-t)\lambda|u(t)|^{2\sigma}u(t)\,dt.
	\end{equation}
	Based \eqref{eq:outgoing-wave} and \eqref{eq:scattering} instead of \eqref{eq:x-wave-formula} or \eqref{eq:y-wave-formula}, the reconstruction of $\lambda, \sigma$ is completely analogous in the outing and scattering settings.
	
	\subsection*{A coefficient depending on both $x$ and $y$}
	
	We further remark that it is possible to deal with the more general case $\lambda= \lambda(x, y)$. Note that the Hermite eigenfunction in \eqref{Hermite} is not crucial: it was chosen only to simplify the formula. In fact, a combination of arguments in Subsection 3.3 and Subsection 4.3 suggests a joint result.
	
	More precisely, for $\lambda=\lambda(x,y)\in W^{1,\infty}(\R^d;\R)$, the forward problem proof still works under the same stronger power condition as in Section 4. For the reconstruction of $\lambda(x,y)$,
	we fix $y_0$ and consider
	$$\phi_{h,\psi,y_0}(x,y)=\psi(x)h^{-m/2}g((y-y_0)/h)$$
	as in Subsection 4.3, but now $\psi$ is a general Schwartz function not necessarily a Hermite eigenfunction.
	The same argument leads to
	\begin{equation}\label{eq:joint}
		\lim_{h\downarrow0}h^{m\sigma-2}\mathcal C^-_{\lambda(x,y),\sigma}(\phi_{h,\psi,y_0})
		=D^-_{\sigma,g}\int_{\R^n}\lambda(x,y_0)|\psi(x)|^p\,dx.
	\end{equation}
	Then we localize at $x_0$ to recover $\lambda(x_0,y_0)$ following the argument in Subsection 3.3.

\end{document}